\documentclass[12pt,a4paper,reqno]{amsart}

\usepackage{amsfonts}
\usepackage{amsmath}
\usepackage{amssymb}
\usepackage{amsthm}
\usepackage{amsxtra}
\usepackage{bbm}
\usepackage{braket}
\usepackage{comment}
\usepackage{extarrows}
\usepackage{latexsym}
\usepackage{mathabx}
\usepackage{mathrsfs}
\usepackage{mathtools} 
\usepackage{verbatim}
\usepackage{xcolor}
\usepackage{cases}

\usepackage[pdfborder={0 0 0}]{hyperref} 
\hypersetup{breaklinks, raiselinks}

\usepackage{bm}
\allowdisplaybreaks[4]

\usepackage[initials,lite]{amsrefs} 
\AtBeginDocument{%
    \def\MR#1{}
}

\usepackage{cleveref}

\usepackage{a4wide}
\theoremstyle{plain}
\newtheorem{Theorem}{Theorem}[section]
\newtheorem{Lemma}[Theorem]{Lemma}

\newtheorem{Proposition}[Theorem]{Proposition}

\theoremstyle{definition}

\newtheorem{Assumptions and Discussion}[Theorem]{Assumptions and Discussion}

\newtheorem{Definition}[Theorem]{Definition}

\newtheorem{Remark}[Theorem]{Remark}

\theoremstyle{remark}

\newtheorem{Setting}[Theorem]{Setting}

\newtheorem*{acknowledgment*}{Acknowledgment}  

\usepackage[shortlabels]{enumitem}
\SetEnumerateShortLabel{a}{\textup{(\alph*)}}
\SetEnumerateShortLabel{A}{\textup{(\Alph*)}}
\SetEnumerateShortLabel{1}{\textup{(\arabic*)}}
\SetEnumerateShortLabel{i}{\textup{(\roman*)}}
\SetEnumerateShortLabel{I}{\textup{(\Roman*)}}  

\def\abs#1{\lvert#1\rvert}

\def\alert#1{{\textcolor{red}{#1}}}

\def\deg{\operatorname{deg}}
\def\depth{\operatorname{depth}}

\def\dim{\operatorname{dim}}  

\def\sign{\operatorname{sign}}

\def\ker{\operatorname{ker}}
\def\KK{{\mathbb K}}

\def\lk{\operatorname{link}} 

\def\NN{{\mathbb N}}

\def\ZZ{{\mathbb Z}}

\newcommand\bdalpha{{\bm \alpha}}

\newcommand\bdbeta{{\bm \beta}}

\newcommand\bdgamma{{\bm \gamma}}

\newcommand\bdx{{\bm x}}

\newcommand\bdy{{\bm y}}

\newcommand\calF{\mathcal{F}}

\newcommand\fraka{\mathfrak{a}}

\newcommand\frakb{\mathfrak{b}}

\newcommand\frakm{\mathfrak{m}}

\newcommand\frakn{\mathfrak{n}}

\newcommand\frakp{\mathfrak{p}}

\newcommand{\Ass}{\operatorname{Ass}}

\newcommand{\im}{\operatorname{im}}

\def\reg{\operatorname{reg}}
     
\begin{document}

\title{$a_i$-invariants of powers of ideals} 
\author{Shi-Xin Tian and Yi-Huang Shen} 
\thanks{2010 {\em Mathematics Subject Classification}.
    Primary 13D45, 
    Secondary 05E40. 
}

\thanks{Keyword: 
    Local cohomology, $a_i$-invariants, symbolic power, fiber product
}

\address{School of Mathematical Sciences, University of Science and Technology of China, Hefei, Anhui, 230026, P.R.~China}
\email{tsx@mail.ustc.edu.cn} 

\address{Key Laboratory of Wu Wen-Tsun Mathematics, Chinese Academy of Sciences, School of Mathematical Sciences, University of Science and Technology of China, Hefei, Anhui, 230026, P.R.~China}
\email{yhshen@ustc.edu.cn (\textrm{Corresponding author})}

\maketitle

\begin{abstract}
    Inspired by the recent work of Lu and O'Rourke, we study the $a_i$-invariants of (symbolic) powers of some graded ideals. The first scenario is  when $I$ and $J$ are two graded ideals in two distinct polynomial rings $R$ and $S$ over a common field $\KK$.  We study the $a_i$-invariants of the powers of the fiber product via the corresponding knowledge of $I$ and $J$.
    The second scenario is when $I_{\Delta}$ is the Stanley-Reisner ideal of a $k$-dimensional simplicial complex $\Delta$ with $k\ge 2$. We investigate the $a_i$-invariants of the symbolic powers of $I_{\Delta}$. 
\end{abstract} 

\section{Introduction}
Let $S=\mathbb{K} [x_{1},\dots, x_{s}]$ and $R=\mathbb{K} [y_{1},\dots, y_{r}]$ be two polynomial rings over a field $\mathbb{K}$ and $T=S \otimes_{\mathbb{K}} R$. Let $I\subseteq S$ and $J\subseteq R$ be two graded ideals. The \emph{fiber product} of $I$ and $J$ is defined by $F=I+J+\mathfrak{m} \mathfrak{n}$, where $\mathfrak{m}$ and $\mathfrak{n}$ are the graded maximal ideals of $S$ and $R$ respectively. One may observe that $\left(S \otimes_{\mathbb{K}} R\right)/\left(I+J+\mathfrak{m}\mathfrak{n}\right)$ can be decomposed as a direct sum of rings $\frac{S}{I}\oplus \frac{R}{J}$. Furthermore, if $I$ and $J$ are edge ideals of two separate graphs, then $I + J + \mathfrak{m}\mathfrak{n}$ corresponds to the edge ideal of the \emph{join} of the graphs. Fiber products of ideals were studied by many authors; c.f.~\cites{arXiv:1910.14140v2,MR3988200,MR3912960,MR3691985}. But little is known about the $a_i$-invariants of $T/(I+J+\mathfrak{m}\mathfrak{n})^k$ yet.

Recall that when $M$ is a finitely generated $S$-module and $0 \leq i \leq \dim(M)$, the \emph{$a_i$-invariant of $M$} is given by
$$
a_i(M) := \max\{t : H_{\mathfrak{m}}^{i}(M)_t \neq 0\},
$$ 
where $H_{\mathfrak{m}}^{i}(M)$ is the \emph{$i$-th local cohomology} module of $M$ with support in $\frakm$. Notice that $a_{\dim(M)}(M)$ is exactly the \emph{$a$-invariant} introduced by Goto and Watanabe in {\cite{MR494707}}. It plays an important role in local duality, since $-a(M)$ is the initial degree of the canonical module of $M$; see, for instance, \cites{MR494707,MR1251956}. 
The $a_i$-invariant also has a close relation with the \emph{Castelnuovo-Mumford regularity}:
$$
\reg \left(M\right):=\max\left\{a_i\left(M\right) + i : 0 \leq i \leq \dim(M)\right\}.
$$
In fact, the $a_i$-invariant takes an important part in the studying its asymptotic behaviour. For example, Herzog, Hoa and Trung {\cite{MR1881017}} proved that if $J$ is a homogeneous of $R$, then $\reg(R/J^n)$ is a linear function of the form $cn+e$ for $n\gg 0$ via investigating $a_i(R/J^n)$. 
Meanwhile, in {\cite{MR2670214}}, Hoa and Trung showed that $a_i(R/J^n)$ is also asymptotically a linear function of $n$.

Let $s$ be a positive integer. By convention, $[s]$ is short for the set $\{1,2,\dots,s\}$.
Let $G$ be a simple graph on $[s]$ considered as a $1$-dimensional simplicial complex and $G^{\prime}$ be obtained from $G$ by adding an isolated vertex $\{s+1\}$. If $I_G\subseteq S$ and $I_{G^{\prime}}\subseteq S[y]$ are their Stanley-Reisner ideals, then $I_{G^{\prime}}=(I_{G},\mathfrak{m} x_{s+1})$.
Based on this observation, in addition to other beautiful results, Lu {\cite{arXiv:1808.07266}} showed the following important result.

\begin{Theorem}
    [{\cite[Theorem 2.8]{arXiv:1808.07266}}]  
    Let $S=\mathbb{K} [x_{1},\dots, x_{s}]$ be a polynomial ring over a field $\mathbb{K}$. Suppose that $\mathfrak{m}=(x_1,\dots,x_s)$ is the graded maximal ideal of $S$ and $y$ is a new variable over $S$. Suppose $I\subseteq S$ is a monomial ideal and $J=(I,\mathfrak{m} y)\subseteq S[y]$.
    \begin{enumerate}[a]
        \item If $i \geq 2$, then 
            $
            a_i(S[y]/ J^{k})=\max\Set{a_i(S / I^{k-t})+t:0 \leq t \leq k-1}.
            $
        \item If $\sqrt{I} \neq \mathfrak{m}$, then
            $
            a_1(S[y] / J^{k})=\max\Set{2k-2,a_1(S / I^{k-t})+t:0 \leq t \leq k-1}.
            $
    \end{enumerate}
\end{Theorem}

Notice that the ideal $(I,\mathfrak{m} y)$ above can also be considered as a fiber product of $I \subseteq S$ and $0 \subseteq \mathbb{K}[y]$. It is then very natural to ask: what can be said towards $a_i(T/ (I+J+\mathfrak{m}\mathfrak{n})^{k})$ in a more general framework? We will answer this in \Cref{Theo}.

Next, we turn our attention to the Stanley-Reisner ideal of simplicial complexes. 
Suppose that $\Delta$ is a simplicial complex on $[s]$ and $I_{\Delta}$ is the Stanley-Reisner ideal of $\Delta$ in $S=\KK[x_1,\dots,x_s]$.
We will deal with the its powers $I_\Delta^n$ and its symbolic powers $I_\Delta^{(n)}$.
The symbolic powers of ideals have a nice geometric description, due to Zariski and Nagata {\cite[Theorem 3.14]{MR1322960}}. The research of related topics has continuously attracted the attention of many researchers; see for instance the recent survey \cite{MR3779569} and the references therein.

Previous related work mainly focuses on symbolic powers of $2$-dimensional squarefree ideals. In {\cites{MR2558862,MR3482347}}, the $a_i$-invariants of symbolic powers of Stanley-Reisner ideals was described explicitly in this case. And in {\cite{arXiv:1808.07266}}, the author proved that for any $1$-dimensional complex $\Delta$ without isolated vertex, one has $a_2(S/I_{\Delta}^{(n)})=a_2(S/I_{\Delta}^{n})$. From these phenomena, it is natural to ask whether $a_{k+1}(S/I_{\Delta}^{(n)})=a_{k+1}(S/I_{\Delta}^{n})$ always holds and under what conditions will $a_{k+1}(S/I_{\Delta}^{(n)})$ be maximal when $\dim(\Delta)=k\geq2$. We will give definite answers to these two questions in \Cref{thm:3.8} and \Cref{Theo4.9}.


\section{$a_i$-invariants of powers of fiber product ideal} 

In this section, we will always assume the following settings.

\begin{Setting}
    \label{setting-section-2}
    Let $S=\mathbb{K}[x_{1},\dots, x_{s}]$ and $R=\mathbb{K}[y_{1},\dots, y_{r}]$ be two polynomial rings over a common field $\mathbb{K}$ and $\mathfrak{m}$ and $\mathfrak{n}$ be the corresponding graded maximal ideals respectively. Let $I \subseteq\mathfrak{m}$ and $J \subseteq \mathfrak{n}$ be two graded ideals and $F=I+J+\mathfrak{m}\mathfrak{n}$ the {fiber product} of $I$ and $J$ in $T=S \otimes_{\mathbb{K}} R$. Fix a positive integer $k$.
\end{Setting}

The aim of this section is to describe the $a_i$-invariants of $T/ F^k$ via the corresponding information of $I$ and $J$. 

Let us start by recalling some pertinent facts of local cohomology and \v{C}ech complex.

\begin{Definition}
    Let $M$ be an $S$-module $M$ and $\fraka$ be an $S$-ideal.
    \begin{enumerate}[a]
        \item Set  
        $$ 
        \Gamma_{\fraka}\left(M\right)\coloneqq\{x \in M : \fraka^tx=0  \text{ for some } t \in \mathbb{N}\}. 
        $$
        Let $ H_{\fraka}^{i}\left(-\right)$ be the $i$-th right derived functor of $\Gamma_{\fraka}\left(-\right)$, namely $ H_{\fraka}^{i}\left(M\right)\coloneqq H^j\left(\Gamma_{\fraka}\left(I^{\bullet}\right)\right)$, in which $I^{\bullet}$ is an injective resolution of $M$. The module $H_{\fraka}^{i}\left(M\right)$ will be called the \emph{$i$-th local cohomology} of $M$ with support in $\fraka$. 
        \item  The module $M$ is called \emph{$\fraka$-torsion} if $\Gamma_{\fraka}\left(M\right)=M$, namely, if each element of $M$ is annihilated by some power of $\fraka$. 
    \end{enumerate}
\end{Definition} 

Next, we collect some well-known facts from \cite[Propositions 7.3 and 7.15]{MR2355715} and \cite[1.2.2 (iv), 2.1.7 Corollary and Exercise 2.1.9]{MR3014449} regarding local cohomology modules.

\begin{Lemma}
    \label{prop:2.2}
    Let $M$ be an $S$-module and $\fraka$ be an $S$-ideal.
    \begin{enumerate}[a]	
        \item  \label{2.2.2} Let $\{M_{\gamma}\}$ be a family of $S$-modules. Then $H_{\fraka}^{j}\left(\oplus_{\gamma}M_{\gamma}\right) \cong \oplus_{\gamma}H_{\fraka}^{j}\left(M_{\gamma}\right)$ for all $j\ge 0$.
        \item  \label{2.2.3} If $S \rightarrow R$ is a ring homomorphism and $N$ is an $R$-module, then $H_{\fraka}^j(N)=H_{\fraka R}^j(N)$.
        \item  \label{2.2.4} Any short exact sequence of $S$-modules $0\rightarrow M \rightarrow N \rightarrow L \rightarrow 0$ induces a long exact sequence of local cohomology modules
            $$
            \cdots\rightarrow H_{\fraka}^{j}\left(M\right) \rightarrow H_{\fraka}^{j}\left(N\right) \rightarrow H_{\fraka}^{j}\left(L\right) \rightarrow H_{\fraka}^{j+1}\left(M\right) \rightarrow \cdots.
            $$
        \item \label{2.2.5} Assume that $M$ is $\frakb$-torsion for some $S$-ideal $\frakb$. Then, $H^j_{\fraka+\frakb}(M) \cong H^j_{\fraka}(M)$ for all $j\ge 0$.
        \item \label{2.2.6} If $M$ is $\fraka$-torsion, then $H_{\fraka}^{j}\left(M\right)=0$ for all $j>0$.
    \end{enumerate}
\end{Lemma}

Our argument afterwards also depends heavily on the computation of local cohomologies in terms of \v{C}ech complexes.

\begin{Definition}
    For elements $m_1,\dots,m_r$ in a commutative ring $R$, set $m_{\sigma}=\prod_{i \in \sigma} m_{i}$ for $\sigma \subseteq[r]$. The \emph{\v{C}ech complex} $\check{\mathcal{C}}^{\bullet}\left(m_{1}, \dots, m_{r}\right)$ is the cochain complex (upper indices increasing from the copy of $R$ sitting in cohomological degree $0$)
    $$
    0 \rightarrow R \rightarrow \bigoplus_{i=1}^{n} R\left[m_{i}^{-1}\right] \rightarrow \dots \rightarrow \bigoplus_{|\sigma|=k} R\left[m_{\sigma}^{-1}\right] \rightarrow \dots \rightarrow R\left[m_{[r]}^{-1}\right] \rightarrow 0,
    $$
    with the map     
    $$
    \partial_{\abs{\sigma}}^i:R[m^{-1}_{\sigma}]\rightarrow R[m^{-1}_{\sigma \cup \{i\}}]
    $$ 
    between the summands in $\check{\mathcal{C}}^{\bullet}\left(m_{1}, \dots, m_{r}\right)$ being $\sign(i, \sigma \cup \{i\})$ times
    the canonical localization homomorphism. 
\end{Definition}

\v{C}ech complex facilitates the computation of local cohomologies.

\begin{Lemma}
    [{\cite[Theorem 13.75]{MR2110098}}]
    \label{thm:2.4}
    The local cohomology of $M$ supported on the ideal $\fraka=(m_1,\dots,m_r)$ in $R$ is the cohomology of the \v{C}ech complex tensored with $M$:
    $$
    H_{\fraka}^{i}(M)=H^{i}\left(M \otimes \check{\mathcal{C}}^{\bullet}\left(m_{1}, \dots, m_{r}\right)\right).
    $$
\end{Lemma}

The following results are also crucial for our argument in this section. 

\begin{Lemma}
    [{\cite [Proposition 4.4]{MR3988200}}]
    \label{prop:2.5}
    Take the assumptions as in \Cref{setting-section-2}. Suppose in addition that $I \subseteq\mathfrak{m}^2$ and $J \subseteq\ \mathfrak{n}^2$. Furthermore, let $H = I + \mathfrak{m}\mathfrak{n}$. 
    \begin{enumerate}[a] 
        \item There is an equality $F^k=H^k+\sum_{i=1}^{k}(\mathfrak{m}\mathfrak{n})^{k-i}J^i$ for each positive integer $k$.
        \item For each $1\leq t\leq k$, denote $G_t = H^k+\sum_{i=1}^{t}(\mathfrak{m}\mathfrak{n})^{k-i}J^i$ and $G_0=H^k$. Then, one has $G_{t-1} \cap(\mathfrak{m} \mathfrak{n})^{k-t} J^{t}=\mathfrak{m}^{k-t+1} \mathfrak{n}^{k-t} J^{t}$ for each $t$.
    \end{enumerate}	
\end{Lemma}  

\begin{Lemma}
    [{\cite[Theorem A4.3]{MR1322960}}]
    \label{Theo2.7} 
    Let $\left(S,\mathfrak{m}\right)$ be a local ring, and let $M$ be a finitely generated $S$-module. We have $H_{\mathfrak{m}}^i(M) = 0$ for $i<\depth(M)$ and for $i > \dim(M)$.
\end{Lemma}

Before presenting the main result of this section, we collect some preliminary results.

\begin{Proposition}
    \label{prop2.8}
    Take the assumptions as in \Cref{setting-section-2}. 
    \begin{enumerate}[a]
        \item\label{prop2.8.1} For any integer $0 \leq t < k$, $a_1\left(S/\mathfrak{m}^tI^{k-t}\right)=a_1\left(S/I^{k-t}\right)$.
        \item\label{prop2.8.2} If $\dim(R)>2$ and $\dim(S)>2$, then $a_1\left(\frac{T}{\mathfrak{m}^k\mathfrak{n}^k} \right)=2k-2$.
        \item\label{prop2.8.3} If $\dim(S)>2$, then $a_1\left(S/I\right)=a_2\left(I\right)$.
        \item \label{prop2.8.4} Set $m_{\sigma\delta}\coloneqq \prod_{i \in \sigma} x_{i}\cdot\prod_{j \in \delta}y_{j}$ for $\sigma \subseteq[s]$ and $\delta \subseteq\ [r]$. Let $F=I+J+\mathfrak{m}\mathfrak{n}\subseteq T$. We use $\partial_j$ to denote the differential map in $\frac{T}{F^k}\otimes \check{\mathcal{C}}^{\bullet}\left(x_{1}, \dots, x_{s}, y_1, \dots, y_r \right)$ at the positions $j$ and $j+1$. Let $\partial_j^0$ and $\partial_j^1$ be the restriction of $\partial_j$ on $\bigoplus_{\substack{\abs{\sigma}=j}}\frac{T}{F^k}[m_{\sigma\emptyset}^{-1}]$ and $\bigoplus_{\substack{\abs{\delta}=j}}\frac{T}{F^k}[m_{\emptyset\delta}^{-1}]$ respectively. Then $\partial_j=\partial_j^1\oplus\partial_j^2$ for each integer $j \geq 1$.
    \end{enumerate}
\end{Proposition}

\begin{proof}
    \begin{enumerate}[a]	
        \item When $0\leq t \leq k-1$, the following short exact sequence
            $$
            0\rightarrow \frac{I^{k-t}}{I^{k-t}\mathfrak{m} ^t} \rightarrow \frac{S}{I^{k-t}\mathfrak{m} ^t} \rightarrow \frac{S}{I^{k-t}} \rightarrow 0
            $$
            induces a long exact sequence
            $$
            \cdots\rightarrow H_{\mathfrak{m} }^{1}\left(\frac{I^{k-t}}{I^{k-t}\mathfrak{m} ^t}\right) \rightarrow H_{\mathfrak{m} }^{1}\left(\frac{S}{I^{k-t}\mathfrak{m} ^t}\right) \rightarrow H_{\mathfrak{m} }^{1}\left(\frac{S}{I^{k-t}}\right) \rightarrow H_{\mathfrak{m} }^{2}\left(\frac{I^{k-t}}{I^{k-t}\mathfrak{m} ^t}\right) \rightarrow \cdots.
            $$
            Since $\frac{I^{k-i}}{I^{k-i}\mathfrak{m} ^i}$ is $\mathfrak{m} $-torsion for $1\leq i \leq k$,  we have $H_{\mathfrak{m}}^{1}\left(\frac{I^{k-i}}{I^{k-i}\mathfrak{m} ^i}\right)=0=H_{\mathfrak{m} }^{2}\left(\frac{I^{k-i}}{I^{k-i}\mathfrak{m} ^i}\right)$ by \Cref{prop:2.2}\ref{2.2.6}. Consequently
            $$
            H_{\mathfrak{m} }^{1}\left(\frac{S}{I^{k-t}\mathfrak{m} ^t}\right) \cong H_{\mathfrak{m} }^{1}\left(\frac{S}{I^{k-t}}\right),
            $$
            and hence $a_1\left(S/ \mathfrak{m}^tI^{k-t}\right)=a_1\left(S/ I^{k-t}\right)$.
            
        \item We will prove this after \Cref{Theo2.6}.
            
        \item The short exact sequence
            $$
            0 \longrightarrow I \longrightarrow S \longrightarrow S/I \longrightarrow 0
            $$
            yields a long exact sequence
            $$
            \cdots \longrightarrow H_{\mathfrak{m}}^1\left(S\right) \longrightarrow H_{\mathfrak{m}}^1\left(S/I\right) \longrightarrow H_{\mathfrak{m}}^2\left(I\right) \longrightarrow H_{\mathfrak{m}}^2\left(S\right) \longrightarrow \cdots .
            $$
            Applying a graded version of \Cref{Theo2.7}, we get $H_{\mathfrak{m}}^1\left(S\right)=H_{\mathfrak{m}}^2\left(S\right)=0$. As a result, $a_1(S/I)=a_2(I)$.
        \item When $j \geq 1$, we have 
            $$
            \frac{T}{F^k}\otimes \check{\mathcal{C}}^{j}\left(x_{1}, \dots, x_{s}, y_1, \dots, y_r \right)=\bigoplus_{\abs{\sigma\cup\delta}=j}\frac{T}{(I+J+\mathfrak{m}\mathfrak{n})^k}[m_{\sigma\delta}^{-1}].
            $$ 
            If both $\sigma$ and $\delta$ are nonempty, then 
            \begin{equation}
                \frac{T}{(I+J+\mathfrak{m}\mathfrak{n})^k}[m_{\sigma\delta}^{-1}]=0.
                \label{eqn:1} 
            \end{equation}
            Therefore, the module $\bigoplus_{\abs{\sigma\cup\delta}=j}\frac{T}{(I+J+\mathfrak{m}\mathfrak{n})^k}[m_{\sigma\delta}^{-1}]$ is simply 
            \[
                \left(\bigoplus_{\substack{\abs{\sigma}=j}}\frac{T}{F^k}[m_{\sigma\emptyset}^{-1}]\right)\oplus\left(\bigoplus_{\substack{\abs{\delta}=j}}\frac{T}{F^k}[m_{\emptyset\delta}^{-1}]\right).
            \]

            Since $\mathfrak{m} T[m_{\sigma\emptyset}^{-1}]=T[m_{\sigma\emptyset}^{-1}]$ and $J+\mathfrak{n}=\mathfrak{n}$, we have
            \begin{align*}
                F^k T[m_{\sigma\emptyset}^{-1}] &= (I+J+\mathfrak{m}\mathfrak{n})^kT[m_{\sigma\emptyset}^{-1}]
               =(I+\mathfrak{n})^kT[m_{\sigma\emptyset}^{-1}]. 
            \end{align*}
            This means that
            \begin{equation}\label{eqn:2}
                \frac{T}{F^k}\left[m_{\sigma\emptyset}^{-1}\right]=\frac{T}{(I+\mathfrak{n})^k}\left[m_{\sigma\emptyset}^{-1}\right]. 
            \end{equation}
            Likewise,
            $$
            \frac{T}{F^k}\left[m_{\emptyset\delta}^{-1}\right]=\frac{T}{(J+\mathfrak{m})^k}\left[m_{\emptyset\delta}^{-1}\right]. 
            $$
            So the \v{C}ech complex at the positions $j$ and $j+1$ can be written as
            \begin{align*}
                \cdots&\rightarrow \left(\bigoplus_{\substack{\abs{\sigma}=j}}\frac{T}{(I+\mathfrak{n})^k}[m_{\sigma\emptyset}^{-1}]\right)\oplus\left(\bigoplus_{\substack{\abs{\delta}=j}}\frac{T}{(J+\mathfrak{m})^k}[m_{\emptyset\delta}^{-1}]\right)\\
                &\stackrel{\partial_{j}}{\longrightarrow} \left(\bigoplus_{\substack{\abs{\sigma}=j+1}}\frac{T}{(I+\mathfrak{n})^k}[m_{\sigma\emptyset}^{-1}]\right)\oplus\left(\bigoplus_{\substack{\abs{\delta}=j+1}}\frac{T}{(J+\mathfrak{m})^k}[m_{\emptyset\delta}^{-1}]\right) \rightarrow \cdots.
            \end{align*}
            Furthermore, when $j>1$,  $\partial_j\left(\frac{T}{(I+\mathfrak{n})^k}[m_{\sigma\emptyset}^{-1}]\right)$ is a subset of
            \begin{align*}
                &\left(\bigoplus_{\substack{\abs{\sigma}=j,\\i\in [s]\backslash \sigma}  }\frac{T}{F^k}\left[m_{\sigma\emptyset}^{-1}x_i^{-1}\right] \right)\oplus \left(\bigoplus_{\substack{\abs{\delta}=j,\\i\in [r]\backslash \delta}}\frac{T}{F^k}\left[m_{\emptyset\delta}^{-1}y_i^{-1}\right]\right)
                =\bigoplus_{\substack{\abs{\sigma}=j+1}}\frac{T}{(I+\mathfrak{n})^k}[m_{\sigma\emptyset}^{-1}].
            \end{align*}
            by \eqref{eqn:1} and \eqref{eqn:2}.
            Then 
            \begin{equation*}
            \im\partial_j^1 \subseteq\bigoplus_{\substack{\abs{\sigma}=j+1}}\frac{T}{(I+\mathfrak{n})^k}[m_{\sigma\emptyset}^{-1}].
            \end{equation*}
            Likewise,
            $$
            \im\partial_j^2 \subseteq\bigoplus_{\substack{\abs{\delta}=j+1}}\frac{T}{(J+\mathfrak{m})^k}[m_{\emptyset\delta}^{-1}].
            $$
            Thus $\partial_j=\partial_j^1\oplus\partial_j^2$ for each integer $j \geq 1$. 
            \qedhere
    \end{enumerate}
\end{proof}

Now, we are ready to present the first main result of this paper.

\begin{Theorem} 
    \label{Theo}
    Take the assumptions as in \Cref{setting-section-2}.  \begin{enumerate}[a]
        \item If $j \geq 2$, then 
        \[
        a_j\left(T/ F^{k}\right)=\max\Set{a_j\left(S/ I^{k-t}\right)+t, a_j\left(R/ J^{k-t}\right)+t: 0 \leq t \leq k-1}.
        \]
        \item \label{Theo(2)} Suppose in addition that $\dim(S)>2$, $\dim(R)>2$, $I \subseteq\mathfrak{m}^2$, $\sqrt{I}\neq \mathfrak{m}$, $J \subseteq \mathfrak{n}^2$ and $\sqrt{J}\neq \mathfrak{n}$. Then,
            $$
            a_1\left(T/ F^{k}\right)=\max\Set{2k-2,a_1\left(S/ I^{k-t}\right)+t, a_1\left(R/ J^{k-t}\right)+t :0 \leq t \leq k-1 }.
            $$
    \end{enumerate}
\end{Theorem}

\begin{proof} 
    Suppose that $j \geq 1$. Then \Cref{thm:2.4} says that 
    $$
    H_{\mathfrak{m}+\mathfrak{n}}^{j}(T/ F^k)=H^{j}\left((T/ F^k) \otimes \check{\mathcal{C}}^{\bullet}\left(x_{1}, \dots, x_{s}, y_1, \dots, y_r\right)\right).
    $$
    Set $m_{\sigma\delta}=\prod_{i \in \sigma} x_{i}\cdot\prod_{i \in \delta}y_{i}\in T$ for $\sigma \subseteq[s]$ and $\delta \subseteq\ [r]$. We have 
    $$
    \frac{T}{F^k}\otimes \check{\mathcal{C}}^{j}\left(x_{1}, \dots, x_{s}, y_1, \dots, y_r \right)=\bigoplus_{\abs{\sigma\cup\delta}=j}\frac{T}{(I+J+\mathfrak{m}\mathfrak{n})^k}[m_{\sigma\delta}^{-1}].
    $$  
    Notice that both $S$ and $R$ have multigraded structures respetively. Hence $T$ will have inherited multigrading, bigrading and standard grading. We will use this fact freely in the following proof. 
    \begin{enumerate}[a]
        \item When $j \geq 2$, we have the following bigraded decomposition via \Cref{prop:2.2}\ref{2.2.3}, \Cref{thm:2.4} and \Cref{prop2.8}\ref{prop2.8.4}:
            \begin{align*}
                H_{\mathfrak{m}+\mathfrak{n}}^{j}\left(\frac{T}{(I+J+\mathfrak{m}\mathfrak{n})^k}\right)
                &=
                \frac{\ker\left(\partial_j\right)}{\im\left(\partial_{j-1}\right)}
                =
                \frac{\ker\left(\partial_j^1\oplus\partial_j^2\right)}{\im\left(\partial_{j-1}^1\oplus\partial_{j-1}^2\right)}\\
                &\cong\frac{\ker\left(\partial_j^1\right)}{\im\left(\partial_{j-1}^1\right)}\oplus\frac{\ker\left(\partial_j^2\right)}{\im\left(\partial_{j-1}^2\right)}\\
                &\cong H_{\mathfrak{m} T}^j\left(\frac{T}{(I+\mathfrak{n})^k}\right) \oplus H_{\mathfrak{n} T}^j\left(\frac{T}{(J+\mathfrak{m})^k}\right)\\ 
                &=H_{\mathfrak{m}}^j\left(\frac{T}{(I+\mathfrak{n})^k}\right) \oplus H_{\mathfrak{n} }^j\left(\frac{T}{(J+\mathfrak{m})^k}\right).
            \end{align*}
            This implies that
            $$
            a_j\left({T}/{(I+J+\mathfrak{m}\mathfrak{n})^k}\right)=\max\left\{a_j\left({T}/{(I+\mathfrak{n})^k}\right),a_j\left({T}/{(J+\mathfrak{m})^k}\right)\right\}.
            $$ 
            As $S$-modules, we have the following bigraded isomorphism:
            \begin{equation*}
                \frac{T}{(I+\mathfrak{n})^k}\cong\bigoplus_{\bdbeta \in \mathbb{N}^r,\\\abs{\bdbeta}<k}\frac{S}{I^{k-\abs{\bdbeta}}}(0,-\abs{\bdbeta}).
            \end{equation*} 
            Then the canonical epimorphism $T \rightarrow S$ induces an isomorphism
            \[
                H_{\mathfrak{m}}^j\left(\frac{T}{(I+\mathfrak{n})^k}\right) \cong H_{\mathfrak{m}}^j\left(\bigoplus_{\bdbeta \in \mathbb{N}^r,\abs{\bdbeta}<k}\frac{S}{I^{k-\abs{\bdbeta}}}(0,-\abs{\bdbeta})\right)
                \cong \bigoplus_{\bdbeta \in \mathbb{N}^r,\abs{\bdbeta}<k}H_{\mathfrak{m}}^j\left(\frac{S}{I^{k-\abs{\bdbeta}}}\right)(0,-\abs{\bdbeta}),
            \]
            via \Cref{prop:2.2} \ref{2.2.2}. Hence 
            $$
            a_j\left({T}/{(I+\mathfrak{n})^k}\right)=\max\left\{a_j\left({S}/{I^{k-t}}\right)+t: 0 \leq t \leq k-1\right\}. 
            $$
            Likewise,
            $$
            a_j\left({T}/{(J+\mathfrak{m})^k}\right)=\max\left\{a_j\left({R}/{J^{k-t}}\right)+t: 0 \leq t \leq k-1\right\}.
            $$
            Therefore, when $j \geq 2$, we arrive at the conclusion that
            $$
            a_j\left(T/ F^{k}\right)=\max\big\{a_j\left(S/ I^{k-t}\right)+t, a_j\left(R/ J^{k-t}\right)+t: 0 \leq t \leq k-1 \big\}. 
            $$

        \item Now we consider the case with $j=1$. The proof will be divided into three steps.  
                
        \medskip
                
        \noindent\textbf{Claim 1:} $a_1\left(T/ F^{k}\right)\geq\max\left\{a_1\left(S/ I^{k-t}\right)+t, a_1\left(R/ J^{k-t}\right)+t :0 \leq t \leq k-1 \right\}$. 
        
        \medskip
            
        Since $\partial_1=\partial_1^1\oplus \partial_1^2$ by \Cref{prop2.8}\ref{prop2.8.4}, one has $\ker\partial_1=\ker\partial_1^1\oplus \ker\partial_1^2$. 
        Let $\partial_0^1$ be the composition of $\partial_0$ with the projection map from $                \bigoplus_{\abs{\sigma\cup\delta}=1}\frac{T}{(I+J+\mathfrak{m}\mathfrak{n})^k}[m_{\sigma\delta}^{-1}]$ to its direct summand $\bigoplus_{\substack{\abs{\sigma}=1}}\frac{T}{(I+J+\mathfrak{m}\mathfrak{n})^k}[m_{\sigma\emptyset}^{-1}])$, and similarly define $\partial_0^2$. It is clear that $\im(\partial_0)\subseteq \im(\partial_0^1) \oplus \im(\partial_0^2)$. Therefore, 
        \[
             \frac{\ker(\partial_1^1)\oplus\ker(\partial_1^2)}{\im(\partial_0^1)\oplus\im(\partial_0^2)} \cong \frac{\ker(\partial_1^1)}{\im(\partial_0^1)}\oplus\frac{\ker(\partial_1^2)}{\im(\partial_0^2)} 
        \]
        is an epimorphic image of $H_{\mathfrak{m}+\frakn}^1(T/F^k) \cong\frac{\ker(\partial_1)}{\im(\partial_0)}$, which in turn implies that
        \[
            a_1(T/ F^{k})
            \geq \max\Set{ l\in \ZZ: \left(\frac{\ker(\partial_1^1)}{\im(\partial_0^1)}\right)_l\neq 0 \text{ or }\left(\frac{\ker(\partial_1^2)}{\im(\partial_0^2)}\right)_l\neq 0 }. 
        \]
        Notice that
        \begin{align*}
            \max\Set{ l\in\ZZ: \left(\frac{\ker(\partial_1^1)}{\im(\partial_0^1)}\right)_l\neq 0}&=\max\Set{ l\in\ZZ: H_{\mathfrak{m}}^j\left(\frac{T}{(I+\mathfrak{n})^k}\right)_l\neq 0}\\
            &=\max\Set{ l\in\ZZ: \bigoplus_{\bdbeta \in \mathbb{N}^r,\abs{\bdbeta}<k}H_{\mathfrak{m}}^1\left(\frac{S}{I^{k-\abs{\bdbeta}}}(0,-\abs{\bdbeta})\right)_l\neq 0}\\
            &=\max\left\{a_1\left(S/ I^{k-t}\right)+t:0 \leq t \leq k-1 \right\}.
        \end{align*}
        Similarly, one has
        \begin{align*}
            \max\Set{l\in\ZZ: \left(\frac{\ker\partial_1^2}{\im\partial_0^2}\right)_l\neq 0 }=\max\Set{a_1\left(R/ J^{k-t}\right)+t:0 \leq t \leq k-1 }. 
        \end{align*}
        Thus, $a_1\left(T/ F^{k}\right)\geq\max\left\{a_1\left(S/ I^{k-t}\right)+t, a_1\left(R/ J^{k-t}\right)+t :0 \leq t \leq k-1 \right\}$, establishing the first claim.

        \medskip
                
        \noindent\textbf{Claim 2:} $a_1\left(T/ F^{k}\right)\geq 2k-2$.
        
        \medskip
        
        It is sufficient to find a bigraded element $u\in \ker(\partial_1)$ such that  its total degree $\deg(u)=2k-2$ and  $u\notin \im(\partial_0)$. For any $v\in T$, let $[v]$, $[v]_{x_i}$ and $[v]_{y_i}$ be the equivalence classes of $v$ in $\frac{T}{F^k}$, $\frac{T}{F^k}[x_i^{-1}]$ and $\frac{T}{F^k}[y_i^{-1}]$ respectively. 
        
        Suppose that $f \in S$ and $g \in R$ with $\deg(f)=\deg(g)=k-1$, then
        \begin{align*}
            [fg]_{x_i} \neq [0]_{x_i} \Longleftrightarrow x_i^lfg \notin \left(I+\mathfrak{n}\right)^k \text{ for } l\geq 0
        \end{align*}
        by the equality \eqref{eqn:2}.
        Since $\deg(g)=k-1$ and $g\in R$, we have $g \notin \mathfrak{n}^k$ and $g \in \mathfrak{n}^p$ for any $1\leq p \leq k-1$. Meanwhile, it is clear that $I^k \subseteq I^{k-1} \subseteq \cdots \subseteq I^2 \subseteq\ I$. Therefore, the above equivalent statements can be further simplified into saying $x_i^lf \notin I \text{ for } l\geq 0$, i.e., $f\notin I:(x_i)^\infty$.
        Consequently, if $[fg]_{x_i} \neq [0]_{x_i}$ for $1 \leq i \leq s$, then $f \notin I:\mathfrak{m}^{\infty}$. Likewise, if $[fg]_{y_i}\neq 0$ for $1\leq i \leq r$, then $g \notin J:\mathfrak{n}^{\infty}$. 
        
        As $\sqrt{I}\ne \mathfrak{m}$, we can find some homogeneous element $f\in S$ of degree $k-1$ such that $f\notin I:\mathfrak{m}^\infty$. Similarly, we can find some homogeneous element $g\in R$ of degree $k-1$ such that $g\notin J:\mathfrak{n}^\infty$. We will verify that $u=\left(\bigoplus_{i=1}^s[fg]_{x_i}\right)\oplus\left(\bigoplus_{i=1}^r[0]_{y_i}\right)$ satisfies the expectation.
        
        To see this, notice first that $u \in \ker(\partial_1)$ and $u$ is bi-homogeneous of degree $(k-1,k-1)$. Consequently, the total degree $\deg(u)=2(k-1)$. Thus, it remains to show that $u \notin \im(\partial_0)$. Suppose for contradiction that there exists an element $h\in T$ such that $\partial_0([h])=\left(\bigoplus_{i=1}^s[fg]_{x_i}\right)\oplus\left(\bigoplus_{i=1}^r[0]_{y_i}\right)$. Without loss of generality, we may assume that $h$ is homogeneous of bidegree $(k-1,k-1)$. Whence, $[h-fg]_{x_i}=[0]_{x_i}$ and $[h]_{y_j}=[0]_{y_j}$ for each $1 \leq i \leq s$ and $1 \leq j \leq r$. Notice that 
        \begin{align*}
            [h-fg]_{x_i}=[0]_{x_i}&\Longleftrightarrow h-fg \in \left(I+\mathfrak{n}\right)^k\colon (x_i)^\infty
        \end{align*}
        for $1\le i\le s$.
        Consequently, we have 
        \[
        h-fg\in (I+\mathfrak{n})^k\colon \mathfrak{m}^\infty \subseteq(I+\mathfrak{n}^k)\colon \mathfrak{m}^\infty.
        \]
        Since the partial degree $\deg_{\bdy}(h-fg)=k-1$, it is clear that $h-fg\notin \mathfrak{n}^k:\mathfrak{m}^{\infty}$ unless $h-fg=0$. So by bigrading, $h-fg \in I:\mathfrak{m}^\infty$. Likewise, we will have $h \in J:\mathfrak{n}^{\infty}$. As a result, $fg=h-(h-fg) \in \left(I:\mathfrak{m}^{\infty}\right)+\left(J:\mathfrak{n}^{\infty}\right)$. Then, again, by bigrading, we will have $f\in I:\mathfrak{m}^{\infty}$ or $g \in J:\mathfrak{n}^{\infty}$, which is a contradiction. And this completes our proof for the second claim.

        \medskip
        
        So far, we have proved that 
        \begin{equation}
            a_1\left(T/F^k\right)\geq \max\left\{2k-2,a_1\left(S/ I^{k-t}\right)+t, a_1\left(R/ J^{k-t}\right)+t :0 \leq t \leq k-1 \right\}.
            \label{4}
        \end{equation} 
        
       \medskip
                
        \noindent\textbf{Claim 3:}  The converse direction of the inequality \eqref{4} also holds.
        
        \medskip       
        
       Let $H=I+\mathfrak{m}\mathfrak{n}$ and $G_t= H^k+\sum_{i=1}^{t}(\mathfrak{m}\mathfrak{n})^{k-i}J^i$ for $0\le t\le k$. Since 
       \[
       G_{t-1} \cap(\mathfrak{m} \mathfrak{n})^{k-t} J^{t}=\mathfrak{m}^{k-t+1} \mathfrak{n}^{k-t} J^{t} 
       \]
       for $1 \leq t \leq k$ by \Cref{prop:2.5}, the following short exact sequence arises
        $$
        0 \rightarrow \frac{(\mathfrak{m}\mathfrak{n})^{k-t}J^t}{\mathfrak{m}^{k-t+1}\mathfrak{n}^{k-t}J^t} \rightarrow \frac{T}{G_{t-1}}\rightarrow \frac{T}{G_t} \rightarrow 0,
        $$ 
        which induces a long exact sequence 
        $$
        \cdots \rightarrow H_{\mathfrak{m}+\mathfrak{n}}^1\left(\frac{T}{G_{t-1}}\right)\rightarrow H_{\mathfrak{m}+\mathfrak{n}}^1\left(\frac{T}{G_{t}}\right) \rightarrow
        H_{\mathfrak{m}+\mathfrak{n}}^2\left(\frac{(\mathfrak{m}\mathfrak{n})^{k-t}J^t}{\mathfrak{m}^{k-t+1}\mathfrak{n}^{k-t}J^t}\right)\rightarrow \cdots 
        $$
        by \Cref{prop:2.2}\ref{2.2.4}. 
        Since $\frac{(\mathfrak{m}\mathfrak{n})^{k-t}J^t}{\mathfrak{m}^{k-t+1}\mathfrak{n}^{k-t}J^t}$ is an $\mathfrak{m} T$-torsion $T$-module, according to \Cref{prop:2.2} \ref{2.2.3} and \ref{2.2.5} we have
        \begin{align*}
        H_{(\mathfrak{m}+\mathfrak{n})T}^2\left(\frac{(\mathfrak{m}\mathfrak{n})^{k-t}J^t}{\mathfrak{m}^{k-t+1}\mathfrak{n}^{k-t}J^t}\right)&\cong H_{\mathfrak{n} T}^2\left(\frac{(\mathfrak{m}\mathfrak{n})^{k-t}J^t}{\mathfrak{m}^{k-t+1}\mathfrak{n}^{k-t}J^t}\right)
        = H_{\mathfrak{n}}^2\left(\frac{(\mathfrak{m}\mathfrak{n})^{k-t}J^t}{\mathfrak{m}^{k-t+1}\mathfrak{n}^{k-t}J^t}\right) \\
        &\cong H_{\mathfrak{n}}^2\left(\frac{\mathfrak{m}^{k-t}}{\mathfrak{m}^{k-t+1}}\otimes \mathfrak{n}^{k-t}J^t\right)
        \cong \frac{\mathfrak{m}^{k-t}}{\mathfrak{m}^{k-t+1}}\otimes H_{\mathfrak{n}}^2\left( \mathfrak{n}^{k-t}J^t\right).
        \end{align*}
        Hence,
        \begin{align*}
            a_1\left(T/ G_t\right) &\leq \max\Set{a_1\left(T/ G_{t-1}\right),a_2\left({(\mathfrak{m}\mathfrak{n})^{k-t}J^t}/{\mathfrak{m}^{k-t+1}\mathfrak{n}^{k-t}J^t}\right)}\\
            &=\max\big\{a_1\left(T/ G_{t-1}\right),a_2\left(\mathfrak{n}^{k-t}J^t\right)+k-t \big\}\\
            &=\max\big\{a_1\left(T/ G_{t-1}\right),a_1\left(R/J^t\right)+k-t \big\}.
        \end{align*}
        The last equivalent holds via \Cref{prop2.8} \ref{prop2.8.1} and \ref{prop2.8.3}.
        Thus we can conclude that
        \begin{align*}
            a_1\left(T/(I+J+\mathfrak{m}\mathfrak{n})^k\right)&=a_1\left(T/ G_k\right)\\
            &\leq\max\big\{a_1\left(T/ G_0\right),a_1\left(R/J^t\right)+k-t:1\leq t\leq k \big\}\\
            &=\max\big\{a_1\left(T/ H^k\right),a_1\left(R/J^t\right)+k-t:1\leq t\leq k \big\}.
        \end{align*}
        Notice that $H=I+\mathfrak{m}\mathfrak{n}$ can also be viewed as the fiber product of $I\subseteq S$ and $(0)\subseteq R$. With a similar argument, we can get
        \begin{align*}
            a_1\left(T/ H^k\right)&\leq\max\big\{a_1\left(T/\mathfrak{m}^k\mathfrak{n}^k\right),a_1\left(S/I^t\right)+k-t:1 \leq t \leq k \big\}\\
            &=\max\Set{2k-2,a_1\left(S/I^t\right)+k-t:1 \leq t \leq k }
        \end{align*}
        by \Cref{prop2.8} \ref{prop2.8.2}.
        These arguments altogether yield
        \begin{equation}\label{5}
            a_1\left(T/F^k\right) \leq\max\left\{2k-2,a_1\left(S/ I^{k-t}\right)+t, a_1\left(R/ J^{k-t}\right)+t :0 \leq t \leq k-1 \right\},
        \end{equation}
        which establishes the third claim. Now, combing the inequalities \eqref{4} with \eqref{5}, we complete the proof.\qedhere
    \end{enumerate}
\end{proof} 

\begin{Remark}
    In \Cref{Theo} \ref{Theo(2)}, if $T=S[y]$ and $J=0$, then the conditions $\dim(S)>2$ and $I\subseteq\mathfrak{m}^2$ can be removed. The proof is only slightly different and will be omitted here.
\end{Remark}  

\section{Top dimensional $a_i$-invariants of $S/I_{\Delta}^{(n)}$}

Let $\Delta$ be a $k$-dimensional simplicial complex over $[s]\coloneqq\{1,2,\dots,s\}$ for some positive integers $k$ and $s$. Suppose that $S=\KK[x_1,\dots,x_s]$ is a polynomial ring over a field $\KK$ and $I_{\Delta}\subseteq S$ is the Stanley-Reisner ideal associated to $\Delta$. The main task of this section is to investigate the $a_{k+1}$-invariant associated to its power $I_\Delta^n$  and its symbolic power $I_\Delta^{(n)}$. The case when $k=1$ has already been considered in {\cite{arXiv:1808.07266}}. There, it was shown that $a_2(S/ I_{\Delta}^{(n)})=a_{2}(S/ I_{\Delta}^{n})$ when $\Delta$ has no isolated vertex. We will generalize this result here by showing that $a_{k+1}(S/ I_{\Delta}^{(n)})=a_{k+1}(S/ I_{\Delta}^{n})$ for any $k$-dimensional simplicial complex $\Delta$. After that, we will characterize when $a_{k+1}(S/ I_{\Delta}^{(n)})$ is maximal. 

Let us start with reviewing some basic notions.
Recall that $\Delta$ is a \emph{simplicial complex} on $[s]$ if $\Delta$ is a collection of subsets of $[s]$ such that if $F \in \Delta $ and $F^{\prime} \subseteq F$ then $F^{\prime} \in \Delta$. Each element $F\in \Delta$ is called a \emph{face} of $\Delta$. The \emph{dimension} of $F$ is defined to be $\dim(F)\coloneqq\abs{F}-1$ and the \emph{dimension} of $\Delta$ is defined to be $\dim(\Delta)\coloneqq \max\Set{\dim(F):F \in \Delta}$. A \emph{facet} is a maximal face of $\Delta$ with respect to inclusion. We will use $\mathcal{F}(\Delta)$ to denote the set of facets of $\Delta$. Meanwhile, a \emph{non-face} of $\Delta$ is a subset $F$ of $[s]$ with $F \notin \Delta$. We will use $\mathcal{N}(\Delta)$ to denote the set of minimal non-faces of $\Delta$.

For any subset $F$ of $[s]$, we set
$$
\bdx_F\coloneqq\prod_{i \in F}x_i\in S.
$$
The \emph{Stanley-Reisner ideal} of $\Delta$ is defined by
$$
I_{\Delta}\coloneqq(\bdx_F:F \in \mathcal{N}(\Delta)) \subseteq S.
$$ 
Let $P_F$ be the prime ideal of $S$ generated by all variables $x_i$ with $i \notin F$. By {\cite[Lemma 1.5.4]{MR2724673}}, the ideal $I_{\Delta}$ has the following primary decomposition
\begin{equation}
I_{\Delta}=\bigcap_{F \in \mathcal{F}(\Delta)} P_{F}.
\label{eqn:pr-dec}
\end{equation}
Recall that the \emph{$n$-th symbolic power} of an ideal $I\subseteq S$ is defined to be
\[
I^{(n)}\coloneqq \bigcap_{\frakp\in \Ass(S/I)}\frakp^n
\]
for $n\ge 1$.
It follows from \eqref{eqn:pr-dec} that the $n$-th {symbolic power} of $I_{\Delta}$ in our situation is precisely
\begin{equation}
I_{\Delta}^{(n)}=\bigcap_{F \in \mathcal{F}(\Delta)}P_{F}^n.
\label{eqn:sym-power-dec}
\end{equation}


In order to describe $a_i(S/I_{\Delta}^{(n)})$, we need to examine the vanishing of $H_\mathfrak{m}^{i}(S/I_{\Delta}^{(n)})_{\bdalpha}$ for $\bdalpha=(\alpha_1,\dots,\alpha_s) \in \mathbb{Z}^{s}$. Due to a formula of Takayama {\cite{MR2165349}}, this problem can be reduced to examining the simplicial homology of the {degree complex}. Set $G_{\bdalpha}\coloneqq\Set{i \in [s]:\alpha_i<0}$. Recall that the \emph{degree complex} $\Delta_{\bdalpha}(I)$ of a monomial ideal $I$ is given by 
\[
    \Delta_{\bdalpha}(I) \coloneqq \Set{F\subseteq [s]\setminus G_{\bdalpha}:\bdx^a \notin IS_{F \cup G_{\bdalpha}}}.
\]
Here, $S_{F \cup G_{\bdalpha}}=S[x_i^{-1}: i \in F \cup G_{\bdalpha}]$ and $\bdx^\bdalpha=x_1^{\alpha_1}\cdots x_s^{\alpha_s}$. 

Related, for each monomial ideal $I$, let $\Delta(I)$ denote the simplicial complex 
\[
\Delta(I)\coloneqq \Set{F \subseteq[s]:\bdx_F \notin \sqrt{I}}.
\]
It is clear that $\Delta(I)=\Delta(\sqrt{I})$. And when $I$ is squarefree, it is exactly the Stanley-Reisner complex of $I$. We also have $\Delta(S)=\emptyset$ since for any $F\in [s]$, $x_F \in \sqrt{S}=S$. 

\begin{Lemma}
    [Takayama]
    \label{3.1}
    Let $I$ be a monomial ideal in $S$ and $\bdalpha$ a vector in $\ZZ^s$. Then
    \[
        \dim_{\mathbb{K}} H_\mathfrak{m}^{i}\left(S/ I\right)_{\bdalpha} =
        \begin{cases}
            \dim_{\mathbb{K}} \widetilde{H}_{i-\abs{G_{\bdalpha}}-1}\left(\Delta_{\bdalpha}(I)\right), & \text{if $G_{\bdalpha} \in \Delta(I)$,} \\
            0, & \text{otherwise.}
        \end{cases}
    \]
\end{Lemma} 

Here, $\widetilde{H}_{i}\left(\Delta_{\bdalpha}(I)\right)$ is the $i$-th reduced simplicial homology group of the complex $\Delta_{\bdalpha}(I)$ over $\KK$.

\begin{Lemma}
    [{\cite[Theorem 4.10]{arXiv:1910.14140v2}}
    \label{Theo2.6}]
    Take the assumptions as in \Cref{setting-section-2}. Suppose in addition that $I\subseteq S$ and $J\subseteq R$ are two monomial ideals. Then, for any $\bdalpha\in \mathbb{Z}^s$, $\bdbeta\in \mathbb{Z}^r$ and $\bdgamma=(\bdalpha,\bdbeta)\in \mathbb{Z}^{s+r}$, we have the following two cases:
    \begin{enumerate}[a]
        \item  if $p=1$ while both $\Delta_{\bdalpha}\left(I^{s-\abs{\bdbeta}}\right)$ and $\Delta_{\bdbeta}\left(I^{s-\abs{\bdalpha}}\right)$ are nonempty, then
        $$
        \dim_{\mathbb{K}} H_{\mathfrak{m}+\mathfrak{n}}^{p}\left(T /(I+J+\mathfrak{m}\mathfrak{n})^{k}\right)_{\bdgamma}=\dim_{\mathbb{K}} H_{\mathfrak{m}}^{p}\left(S / I^{k-\abs{\bdbeta}}\right)_{\bdalpha}+\dim_{\mathbb{K}} H_{\mathfrak{n}}^{p}\left(R / J^{k-\abs{\bdalpha}}\right)_{\bdbeta}+1;
        $$

        \item otherwise,
        $$
        \dim_{\mathbb{K}} H_{\mathfrak{m}+\mathfrak{n}}^{p}\left(T /(I+J+\mathfrak{m}\mathfrak{n})^{k}\right)_{\bdgamma}=\dim_{\mathbb{K}} H_{\mathfrak{m}}^{p}\left(S / I^{k-\abs{\bdbeta}}\right)_{\bdalpha}+\dim_{\mathbb{K}} H_{\mathfrak{n}}^{p}\left(R / J^{k-\abs{\bdalpha}}\right)_{\bdbeta}.
        $$ 
    \end{enumerate}
    Here, $|\bdalpha|=\sum_{i=1}^s \alpha_i$ for $\bdalpha=(\alpha_1,\dots,\alpha_s)$ and one can similarly define $|\bdbeta|$.
\end{Lemma} 

As a quick application of the above two lemmas, we finish the proof of \Cref{prop2.8}.

\begin{proof}
    [Proof of \Cref{prop2.8} \ref{prop2.8.2}]
    Notice that the ideal $\mathfrak{m}^k\mathfrak{n}^k\subseteq T$ is the fiber product of $I=\left(0\right)\subseteq S$ and $J=\left(0\right) \subseteq R$. Now, take arbitrary $\bdalpha\in \mathbb{Z}^s$ and $\bdbeta\in \mathbb{Z}^r$. 
     
    Firstly, we consider $H_{\mathfrak{m}}^{1}\left(S / I^{k-\abs{\bdbeta}}\right)_{\bdalpha}$. If $\abs{\bdbeta}<k$, then $I^{k-\abs{\bdbeta}}=0$. Hence $H_{\mathfrak{m}}^{1}\left(S / I^{k-\abs{\bdbeta}}\right)_{\bdalpha}=H_{\mathfrak{m}}^{1}\left(S\right)_{\bdalpha}=0$ by \Cref{Theo2.7}. When $\abs{\bdbeta}\geq k$, we have $I^{k-\abs{\bdbeta}}=S$. Then $H_{\mathfrak{m}}^{1}\left(S / I^{k-\abs{\bdbeta}}\right)_{\bdalpha}=H_{\mathfrak{m}}^{1}\left(0\right)_{\bdalpha}=0$. Thus for any $\bdalpha$ and $\bdbeta$, $H_{\mathfrak{m}}^{1}\left(S / I^{k-\abs{\bdbeta}}\right)_{\bdalpha}=0$.  
    Likewise, $H_{\mathfrak{n}}^{1}\left(R / J^{k-\abs{\bdalpha}}\right)_{\bdbeta}=0$. 
    
    So, according to \Cref{Theo2.6}, if $H_{\mathfrak{m}+\mathfrak{n}}^{1}\left(T /(I+J+\mathfrak{m}\mathfrak{n})^{k}\right)_{(\bdalpha,\bdbeta)}\ne 0$, then both $\Delta_{\bdalpha}\left(I^{k-\abs{\bdbeta}}\right)$ and $\Delta_{\bdbeta}\left(I^{k-\abs{\bdalpha}}\right)$ are nonempty. In the following, we will suppose that this is the case.
    
    Notice that if $\abs{\bdbeta}\geq k$, then $I^{k-\abs{\bdbeta}}=S$. Whence, for any  $F\subseteq [s]\setminus G_{\bdalpha}$, we have $\bdx^{\bdalpha} \in S_{F \cup G_{\bdalpha}}$, which implies that $F \notin \Delta_{\bdalpha}\left(S\right)$ by definition. So $\Delta_{\bdalpha}\left(I^{k-\abs{\bdbeta}}\right)=\emptyset$, contradicting to our previous assumption. Thus $\abs{\bdbeta} \leq k-1$  and similarly $\abs{\bdalpha} \leq k-1$. Consequently, $a_1\left(\frac{T}{\mathfrak{m}^k\mathfrak{n}^k} \right)\leq2k-2$. 
    
    On the other hand, let $\bdalpha=(k-1,0,\dots,0)\in \mathbb{Z}^s$ and $\bdbeta=(k-1,0,\dots,0)\in \mathbb{Z}^r$. we have $[s]\setminus\{1\} \in \Delta_{\bdalpha}\left(I^{k-\abs{\bdbeta}}\right)$ since $x_1^{k-1} \notin (0)S_{[s]\setminus\{1\}}$. So $\Delta_{\bdalpha}\left(I^{k-\abs{\bdbeta}}\right)$ is nonempty. Likewise, $\Delta_{\bdbeta}\left(J^{k-\abs{\bdalpha}}\right)$ is nonempty. Thus, $H_{\mathfrak{m}+\mathfrak{n}}^{1}\left(T /(\mathfrak{m}\mathfrak{n})^{k}\right)_{(\bdalpha,\bdbeta)}\neq 0$ via  \Cref{Theo2.6}, meaning $a_1\left(\frac{T}{\mathfrak{m}^k\mathfrak{n}^k} \right)\ge 2k-2$. So $a_1\left(\frac{T}{\mathfrak{m}^k\mathfrak{n}^k} \right)=2k-2$, completing the proof.
\end{proof} 

The following lemma gives a precise description of $\Delta_{\bdalpha}(I_{\Delta}^{(n)})$.  

\begin{Lemma}
    [{\cite[Lemma 1.3]{MR3482347}}]
    \label{3.2}
    Assume that $G_{\bdalpha} \in \Delta$ for some $\bdalpha \in \mathbb{Z}^r$. Then
    \[
        \mathcal{F}(\Delta_{\bdalpha}(I_{\Delta}^{(n)}))=\Set{F\in \mathcal{F}(\lk_{\Delta}(G_{\bdalpha})): \sum_{i \notin F \cup G_{\bdalpha}}a_i \leq n-1}.
    \]
\end{Lemma}

The concept of {monomial} {localization} was introduced in \cite{MR3011344} as a simplification of the localization. Fix a subset $F \subseteq[s]$.  Let $\pi_F : S \rightarrow \mathbb{K}[x_i : i \in [s] \setminus F]$ be the $\mathbb{K}$-algebra homomorphism sending $x_i$ to $x_i$ for $ i \in [s] \setminus F$ and $x_i$ to 1 for $i \in F$. The image of a monomial ideal $I$ of $S$ under the map $\pi_F$ is called the \emph{monomial localization} of $I$ with respect to $F$ and will be denoted by $I[F]$. It is clear that if $I$ and $J$ are both monomial ideals of $S$, then $(IJ)[F] = I[F]J[F]$ and $(I\cap J)[F] = I[F]\cap J[F]$.

The degree complex can be expressed using the monomial localization as follows.

\begin{Lemma}
    [{\cite[Lemmas 1.4, 1.5]{arXiv:1808.07266} and \cite[Lemma 1.3]{MR2558862}}]
    \label{2.3} 
    Let $I$ be a monomial ideal in $S=\mathbb{K} [x_{1},\dots, x_{s}]$ and $\bdalpha=(\alpha_1,\dots,\alpha_s)$ be a vector in $\mathbb{Z}^s$. Define $\bdalpha_+$ to be the non-negative part of $\bdalpha$, i.e., $\bdalpha_+\coloneqq (\alpha_1',\dots,\alpha_s')$ where $\alpha_i'=\max(0,\alpha_i)$ for each $i$.
    \begin{enumerate}[a]	
        \item \label{Lem4.1} 
            $\Delta_{\bdalpha}(I)$ is a subcomplex of $\Delta(I)$. Moreover, if $I$ has no embedded associated prime and $\bdalpha \in \mathbb{N}^s$, then $\calF(\Delta_{\bdalpha}(I))\subseteq \calF(\Delta(I))$.
        \item\label{Lem4.2} 
            $\Delta_{\bdalpha}(I)=\{F \subseteq[s]\setminus G_{\bdalpha} : \bdx^{\bdalpha_+} \notin I[F \cup G_{\bdalpha} ]S\}$. 
        \item \label{Lem4.3}
            If $G_{\bdalpha} \neq \emptyset$, then $\Delta_{\bdalpha}(I)=\lk_{\Delta_{\bdalpha_+}(I)}(G_{\bdalpha})$.
    \end{enumerate}
\end{Lemma}  

\begin{Lemma}
    \label{Lem4.4}
    Let $\Delta$ be a $k$-dimensional complex on $[s]$. If $F \in \Delta$ with $\dim(F)=k$, then
    $$
    I_{\Delta}[F]=(x_i : i\in[s]\setminus F).
    $$ 
\end{Lemma}

\begin{proof}
    As $1 \notin I_{\Delta}[F]$, it follows that $I_{\Delta}[F] \subseteq (x_i : i\in [s] \setminus F)$. Conversely, for each $i \in [s] \setminus F$, $F\cup \{i\} \notin \Delta$ since $\dim(\Delta)=k$. So $\bdx_{F'\cup\{i\}} \in I_{\Delta}$ for some $F'\subseteq F$, implying that $x_i \in I_{\Delta}[F]$. Since this holds for any $i\in [s]\setminus F$, we have the converse containment $I_{\Delta}[F] \supseteq (x_i : i\in [s] \setminus F)$.
\end{proof}

\begin{Proposition}
    \label{Prop4.5}
    Let $\Delta$ be a $k$-dimensional complex on $[s]$. For each $\bdalpha=(\alpha_1,\dots,\alpha_s)\in\mathbb{N}^s$ and each $k$-dimensional simplex $F\subseteq [s]$, the following statements are equivalent:
    \begin{enumerate}[a]
        \item \label{Prop4.5-a} $F \in \Delta_{\bdalpha}(I_{\Delta}^n)$;
        \item \label{Prop4.5-b} $F\in \Delta$ and $\sum_{i \in [s]\setminus F} \alpha_i \leq n-1.$
    \end{enumerate}
\end{Proposition}

\begin{proof}
    Suppose first that $F\in \Delta_{\bdalpha}(I_{\Delta}^n)$. Since $G_{\bdalpha}=\emptyset$, according to \Cref{2.3}\ref{Lem4.1} we have $F \in \Delta_{\bdalpha}(I_{\Delta}^n) \subseteq\Delta(I_{\Delta}^n)=\Delta(I_{\Delta})=\Delta$. It follows from \Cref{2.3}\ref{Lem4.2} and \Cref{Lem4.4} that 
    \begin{equation*}
        \bdx^{\bdalpha}\notin I_{\Delta}^n[F]S
        =\left(I_{\Delta}[F]\right)^nS
        =(x_i :i\in[s]\setminus F)^nS. 
    \end{equation*}
    Hence $\sum_{i \in [s]\setminus F} \alpha_i \leq n-1$ and this proves \ref{Prop4.5-a}$\Rightarrow$\ref{Prop4.5-b}. 
    
    Conversely, suppose that $F \in \Delta$ and $\sum_{i \in [s]\setminus F} \alpha_i \leq n-1$. Then by \Cref{Lem4.4}, 
    \begin{equation*}
        \bdx^{\bdalpha}\notin (x_i :i\in[s]\setminus F)^n=\left(I_{\Delta}[F]\right)^nS
        =I_{\Delta}^n[F]S.
    \end{equation*}
    Thus $F \in \Delta_{\bdalpha}(I_{\Delta}^n)$ via \Cref{2.3}\ref{Lem4.2},  which proves \ref{Prop4.5-b}$\Rightarrow$\ref{Prop4.5-a}.
\end{proof}

\begin{Lemma}
    \label{Lem4.8}
    If $\Delta$ is a complex on $[s]$ and $I_\Delta$ is the Stanley-Reisner ideal in $S=\mathbb{K} [x_{1},\dots, x_{s}]$ over a field $\mathbb{K}$, then 
    $
    \Delta=\Delta\left(I_{\Delta}^{n}\right)=\Delta(I_{\Delta}^{(n)})
    $ 
    for all positive integer $n$.
\end{Lemma}

\begin{proof}
    It follows from \eqref{eqn:pr-dec} and \eqref{eqn:sym-power-dec} that
    \begin{align*}	
        \sqrt{I_{\Delta}^{(n)}}=\sqrt{\bigcap_{F \in \mathcal{F}(\Delta)}P_{F}^n}=\bigcap_{F \in \mathcal{F}(\Delta)}\sqrt{P_{F}^n}=\bigcap_{F \in \mathcal{F}(\Delta)}\sqrt{P_{F}}=\sqrt{\bigcap_{F \in \mathcal{F}(\Delta)}P_{F}}=\sqrt{I_{\Delta}}=\sqrt{I_{\Delta}^{n}}.
    \end{align*}
    Therefore,
    \[
    \Delta(I_{\Delta}^{(n)}) =\Delta\left(\sqrt{I_{\Delta}^{(n)}}\right)=\Delta\left(\sqrt{I_{\Delta}^{n}}\right)=\Delta\left(I_{\Delta}^{n}\right),
    \]
    and they agree with $\Delta(\sqrt{I_\Delta})=\Delta(I_\Delta)=\Delta$.
\end{proof}  

Recall that the \emph{pure $i$-th skeleton} of $\Delta$ is the pure simplicial complex $\Delta^{(i)}$ whose facets are the faces $F$ of $\Delta$ with $\dim(F)=i$.

\begin{Proposition}
    \label{Lem4.7}
    Let $\Delta$ be a $k$-dimensional simplicial complex over $[s]$. For any $\bdalpha \in \mathbb{Z}^s$ with $G_{\bdalpha}\in \Delta$, we have $\Delta_{\bdalpha}(I_{\Delta}^n)^{(k)}=\Delta_{\bdalpha}(I_{\Delta}^{(n)})^{(k)}$.
\end{Proposition}

\begin{proof}
    If $G_{\bdalpha}\ne \emptyset$, then $\dim (\Delta_{\bdalpha}(I_{\Delta}^{(n)}))\le \dim(\lk_{\Delta}(G_{\bdalpha}))<\dim(\Delta)=k$ by \Cref{3.2}. 
    Similarly, if $G_{\bdalpha}\ne \emptyset$, then $\dim(\Delta_\bdalpha(I_\Delta^n))< \dim(\Delta_{\bdalpha_+}(I_\Delta^n))\le \dim(\Delta(I_\Delta^n))=\dim(\Delta)=k$ by
    \Cref{2.3} \ref{Lem4.1} and \ref{Lem4.3}. Therefore, $\Delta_{\bdalpha}(I_{\Delta}^n)^{(k)}=\{\emptyset\}=\Delta_{\bdalpha}(I_{\Delta}^{(n)})^{(k)}$. 
    
    When $G_{\bdalpha}=\emptyset$, then $\lk_{\Delta}(G_{\bdalpha})=\Delta$. Now, for each $F\in\Delta$ with $\dim(F)=k$, we have
    \begin{align*}
        F \in \Delta_{\bdalpha}(I_{\Delta}^{(n)}) \Longleftrightarrow \text{$F$ is a facet of $\Delta$ and $\sum_{i \in [s]\setminus F} \alpha_i \leq n-1$} \Longleftrightarrow F \in \Delta_{\bdalpha}(I_{\Delta}^{n}).
    \end{align*}
    The first equivalence comes from \Cref{3.2} and the second comes from \Cref{Prop4.5}. So we can conclude safely with $\Delta_{\bdalpha}(I_{\Delta}^n)^{(k)}=\Delta_{\bdalpha}(I_{\Delta}^{(n)})^{(k)}$.
\end{proof}

Now, we can state the second main result of this paper.

\begin{Theorem}
\label{thm:3.8}
    Let $\Delta$ be a $k$-dimensional complex on $[s]$. If $I_\Delta$ is the Stanley-Reisner ideal in the polynomial ring $S=\mathbb{K} [x_{1},\dots, x_{s}]$ over a field $\mathbb{K}$, then $a_{k+1}(S/ I_{\Delta}^{(n)})=a_{k+1}\left(S/ I_{\Delta}^{n}\right)$ for all $n \geq 1$.
\end{Theorem}

\begin{proof}
    We have already seen that $\Delta=\Delta\left(I_{\Delta}^{n}\right)=\Delta(I_{\Delta}^{(n)})$ by \Cref{Lem4.8}.
    Now, take an arbitrary $\bdalpha \in \mathbb{Z}^s$. If $G_{\bdalpha}\notin \Delta$, then
    $$
    H_{\mathfrak{m}}^{k+1}(S / I_{\Delta}^{(n)})_{\bdalpha}=H_{\mathfrak{m}}^{k+1}\left(S / I_{\Delta}^{n}\right)_{\bdalpha}=0
    $$
    by \Cref{3.1}. Thus, we may assume instead that $G_{\bdalpha}\in \Delta$. 
    In this case, we claim that
    \begin{equation}
        \Delta_{\bdalpha}(I_{\Delta}^n)^{(k-\abs{G_{\bdalpha}})}=\Delta_{\bdalpha}(I_{\Delta}^{(n)})^{(k-\abs{G_{\bdalpha}})}.
        \label{eqn:equal-1}
    \end{equation}
    Notice that 
    \[
    \Delta_{\bdalpha_+}(I_{\Delta}^n)^{(k)}=\Delta_{\bdalpha_+}(I_{\Delta}^{(n)})^{(k)}.
    \] 
    by \Cref{Lem4.7}. Therefore, \eqref{eqn:equal-1} holds when $|G_\bdalpha|=0$. In the following, we will assume additionally that $|G_\bdalpha|\ge 1$. Now, $\Delta_{\bdalpha}(I_{\Delta}^{(n)})^{(k-\abs{G_{\bdalpha}})}$ is actually a simplicial complex over $[s]\setminus G_\bdalpha$ by \Cref{3.2}. Meanwhile, $\Delta_{\bdalpha}(I_{\Delta}^{(n)})^{k-\abs{G_{\bdalpha}}}$ is also a simplicial complex over $[s]\setminus G_\bdalpha$ by \Cref{2.3}\ref{Lem4.3}. Hence, to establish \eqref{eqn:equal-1} in this situation, we will take an arbitrary ($k-\abs{G_{\bdalpha}}$)-dimensional face $A\in \Delta$ such that $A\cap G_\bdalpha=\emptyset$. Now,
    \begin{align}
        A \in \Delta_{\bdalpha}(I_{\Delta}^{(n)})&\Longleftrightarrow A\in \calF(\Delta_{\bdalpha}(I_{\Delta}^{(n)})) \notag\\
        &\Longleftrightarrow \text{$A \in \mathcal{F}(\lk_{\Delta}(G_{\bdalpha}))$ and $\sum_{i\notin A\cup G_{\bdalpha}}\alpha_i \leq n-1$} \label{eqv-1}\\ 
        &\Longleftrightarrow \text{$A \cup G_{\bdalpha}\in\Delta$ and $\sum_{i\notin A\cup G_{\bdalpha}}\alpha_i \leq n-1$} \notag\\
        &\Longleftrightarrow A\cup G_{\bdalpha} \in \Delta_{\bdalpha_+}(I_{\Delta}^n)\label{eqv-2}\\
        &\Longleftrightarrow A \in \Delta_{\bdalpha}(I_{\Delta}^{n}).\label{eqv-3}
    \end{align}
    The equivalences in \eqref{eqv-1}, \eqref{eqv-2} and \eqref{eqv-3} come from \Cref{3.2}, \Cref{Prop4.5} and \Cref{2.3}\ref{Lem4.3} respectively. And this establishes the equality in \eqref{eqn:equal-1}.
    
    Notice that $\dim (\Delta_{\bdalpha}(I_{\Delta}^n)^{(k-\abs{G_{\bdalpha}})})=\dim (\Delta_{\bdalpha}(I_{\Delta}^{(n)})^{(k-\abs{G_{\bdalpha}})})=k-\abs{G_{\bdalpha}}$. Consequently, the boundaries
    $$
    B_{k-\abs{G_{\bdalpha}}}\left(\Delta_{\bdalpha}(I_{\Delta}^{n})^{(k-\abs{G_{\bdalpha}})}\right)=B_{k-\abs{G_{\bdalpha}}}\left(\Delta_{\bdalpha}(I_{\Delta}^{(n)})^{(k-\abs{G_{\bdalpha}})}\right)=0.
    $$
    Thus, by \eqref{eqn:equal-1}, the simplicial homologies
    \begin{align*}
       H_{k-\abs{G_{\bdalpha}}} \left(\Delta_{\bdalpha}(I_{\Delta}^{n});\mathbb{K}\right)&=\frac{Z_{k-\abs{G_{\bdalpha}}}\left(\Delta_{\bdalpha}(I_{\Delta}^{n})^{(k-\abs{G_{\bdalpha}})}\right)}{B_{k-\abs{G_{\bdalpha}}}\left(\Delta_{\bdalpha}(I_{\Delta}^{n})^{(k-\abs{G_{\bdalpha}})}\right)}=Z_{k-\abs{G_{\bdalpha}}}\left(\Delta_{\bdalpha}(I_{\Delta}^{n})^{(k-\abs{G_{\bdalpha}})}\right)\\
        &=Z_{k-\abs{G_{\bdalpha}}}\left(\Delta_{\bdalpha}(I_{\Delta}^{(n)})^{(k-\abs{G_{\bdalpha}})}\right)=\frac{Z_{k-\abs{G_{\bdalpha}}}\left(\Delta_{\bdalpha}(I_{\Delta}^{(n)})^{(k-\abs{G_{\bdalpha}})}\right)}{B_{k-\abs{G_{\bdalpha}}}\left(\Delta_{\bdalpha}(I_{\Delta}^{(n)})^{(k-\abs{G_{\bdalpha}})}\right)}\\
        &=
        H_{k-\abs{G_{\bdalpha}}} \left(\Delta_{\bdalpha}(I_{\Delta}^{(n)});\mathbb{K}\right),
    \end{align*}
    and consequently,
    $$
    \widetilde{H}_{k-\abs{G_{\bdalpha}}} \left(\Delta_{\bdalpha}(I_{\Delta}^{n});\mathbb{K}\right)=\widetilde{H}_{k-\abs{G_{\bdalpha}}} \left(\Delta_{\bdalpha}(I_{\Delta}^{(n)});\mathbb{K}\right).
    $$
    This equality together with \Cref{3.1} will yield
    $$
    H_{\mathfrak{m}}^{k+1}(S / I_{\Delta}^{(n)})_{\bdalpha} \neq 0 \Longleftrightarrow H_{\mathfrak{m}}^{k+1}\left(S / I_{\Delta}^{n}\right)_{\bdalpha} \neq 0,
    $$ 
    which finishes the proof.
\end{proof} 

In the rest of this paper, we will examine when $a_{k+1}(S/ I_{\Delta}^{(n)})$ is maximal. The following lemma allows us to clarify some details.
\begin{Lemma}
    \label{2.9}
    Let $\Delta$ be a $k$-dimensional complex on $[s]$ and $I_\Delta$ the Stanley-Reisner ideal in the polynomial ring $S=\mathbb{K} [x_{1},\dots, x_{s}]$ over a field $\mathbb{K}$. 
    Suppose that $\bdalpha=(\alpha_1,\dots,\alpha_s)\in \mathbb{Z}^s$ such that $\widetilde{H}_{k-\abs{G_{\bdalpha}}}(\Delta_{\bdalpha}(I_{\Delta}^{(n)}))\neq0$ and $G_{\bdalpha}\in \Delta$. Then, $\alpha_i\leq n-1$ for each $i \in [s]$. 
\end{Lemma}

\begin{proof}
    Suppose that this is not true. Without loss of generality, we may assume that $\alpha_1\ge n$. 
    Let $A_1,\dots,A_l$ be the complete list of $(k-\abs{G_{\bdalpha}})$-dimensional faces in $\Delta_{\bdalpha}(I_{\Delta}^{(n)})$. 
    Then, according to \Cref{3.2}, we have $1 \in A_i$ for each $1 \leq i \leq l$. Hence,
    $\Delta_{\bdalpha}(I_{\Delta}^{(n)})^{(k-|G_\bdalpha|)}$ is a cone. Thanks to {\cite [Theorem 8.2]{MR755006}}, 
    $\widetilde{H}_{k-\abs{G_{\bdalpha}}}(\Delta_{\bdalpha}(I_{\Delta}^{(n)})^{(k-|G_\bdalpha|)})=0$. Since $\Delta_{\bdalpha}(I_{\Delta}^{(n)})$ is $(k-|G_\bdalpha|)$-dimensional, this implies that
    $\widetilde{H}_{k-\abs{G_{\bdalpha}}}(\Delta_{\bdalpha}(I_{\Delta}^{(n)}))=0$, a contradiction to our assumption. 
\end{proof}

Finally, we are ready to present the last main result of this paper.

\begin{Theorem}
    \label{Theo4.9}
    Let $\Delta$ be a $k$-dimensional complex on $[s]$ and $I_\Delta$ the Stanley-Reisner ideal in the polynomial ring $S=\mathbb{K} [x_{1},\dots, x_{s}]$ over a field $\mathbb{K}$.
    Then  
    $$
    a_{k+1}(S/ I_{\Delta}^{(n)}) \leq (k+2)(n-1)
    $$ 
    for each positive integer $n$. Furthermore, the following statements are equivalent for $n\ge 2$: 
    \begin{enumerate}[a]
        \item \label{Theo4.9-a}
            $a_{k+1}(S/ I_{\Delta}^{(n)})=(k+2)(n-1)$;
        \item \label{Theo4.9-b}
            there exists a subset $B=\{p_1,\dots,p_{k+2}\}\subseteq [s]$ such that $\mathcal{F}(\Delta|_B)$ is a $k$-dimensional sphere, namely
            $$
            \mathcal{F}(\Delta|_B)=\Set{B\setminus\{p_i\}:1 \leq i \leq k+2}.
            $$
    \end{enumerate} 
\end{Theorem}

\begin{proof}
    Take arbitrary $\bdalpha \in \mathbb{Z}^s$ with $H_{\mathfrak{m}}^{k+1}(S / I_{\Delta}^{(n)})_{\bdalpha} \neq 0$. According to \Cref{3.1}, this simply means that $\widetilde{H}_{k-|G_\bdalpha|} (\Delta_{\bdalpha}(I_{\Delta}^{(n)})) \neq 0$ and $G_{\bdalpha}\in \Delta$. Since 
    \[
    \dim(\Delta_{\bdalpha}(I_{\Delta}^{(n)}))\le \dim(\Delta)-|G_\bdalpha|= k-|G_\bdalpha|
    \]
    by \Cref{3.2}, the nonvanishing of 
    $\widetilde{H}_{k-|G_\bdalpha|} (\Delta_{\bdalpha}(I_{\Delta}^{(n)}))$
    implies particularly that $\dim(\Delta_{\bdalpha}(I_{\Delta}^{(n)}))=k-|G_\bdalpha|$.  Therefore, we can take some $F \in \Delta_{\bdalpha}(I_{\Delta}^{(n)})$ with $\dim(F)=k-|G_\bdalpha|$. Notice that $F\cap G_\bdalpha=\emptyset$, $\sum_{i \notin F\cup G_{\bdalpha}}\alpha_i\leq n-1$ and $\sum_{i \in F}\alpha_i \leq \abs{F}(n-1)$ by \Cref{3.2} and \Cref{2.9} respectively. Henceforth,
    \begin{align*}
        |\bdalpha|&=\sum_{i \in F}\alpha_i+\sum_{i \notin F\cup G_{\bdalpha}}\alpha_i+\sum_{i \in G_{\bdalpha}}\alpha_i\\ 
        &\leq \abs{F}(n-1)+(n-1)-\abs{G_{\bdalpha}}\\ 
        &=(k-|G_\bdalpha|+2)(n-1)-\abs{G_{\bdalpha}}\\
        &\leq (k+2)(n-1),
    \end{align*}
    establishing the expected upper bound. 
    It remains to prove the equivalence of \ref{Theo4.9-a} and \ref{Theo4.9-b} when $n\ge 2$.
        
    \textbf{\ref{Theo4.9-a}$\Rightarrow$\ref{Theo4.9-b}:}
    Suppose that $a_{k+1}(S/ I_{\Delta}^{(n)})=(k+2)(n-1)$. Then, we can find some $\bdalpha\in\ZZ^s$ such that $ H_{\mathfrak{m}}^{k+1}(S / I_{\Delta}^{(n)})_{\bdalpha} \neq 0$ and $\abs{\bdalpha}=(k+2)(n-1)$. 
    From the above argument, we can see that $\abs{G_{\bdalpha}}=0$, i.e., $G_\bdalpha=\emptyset$. Furthermore, there exists some $F\in \Delta_{\bdalpha}(I_{\Delta}^{(n)})$ such that $\dim(F)=k$. And since $\widetilde{H}_{k} (\Delta_{\bdalpha}(I_{\Delta}^{(n)}))\ne 0$, $\Delta_{\bdalpha}(I_{\Delta}^{(n)})\ne \braket{F}$.
    Assume that $F=[k+1]$ for convenience. Now, $\sum_{i \in [k+1]}\alpha_i=(k+1)(n-1)$ and $\sum_{i \notin [k+1]}\alpha_i=n-1$. Hence $\alpha_1=\alpha_2=\dots=\alpha_{k+1}=n-1$ by \Cref{2.9}. Our task is then reduced to describing $\alpha_j$ when $k+1<j\leq s$.
    
    We claim that there exists precisely one $j$ with $\alpha_j>0$ and $k+1<j \leq s$. If this is not true, we may assume that $\alpha_{k+2},\alpha_{k+3}>0$. As $\Delta_{\bdalpha}(I_{\Delta}^{(n)})\ne \braket{F}$, we may find some facet $G$ of $\Delta_{\bdalpha}(I_{\Delta}^{(n)})$ with $G \neq F$. Since $\dim(G)\le \dim(F)=k$, either $|F \setminus G|\ge 2$ or $|\{k+2,k+3\}\setminus G|\ge |F \setminus G|= 1$.
    In both cases, we have $\sum_{i\notin G}\alpha_i>n-1$, which implies that $G\notin \Delta_{\bdalpha}(I_{\Delta}^{(n)})$ by \Cref{3.2}, a contradiction to the choice of $G$. 

    Therefore, we may assume that $\alpha_{k+2}=n-1$ and $\alpha_h=0$ for $k+2<h\leq s$. 
    An argument as in the previous paragraph also shows that $G\subseteq B\coloneqq[k+2]$ and $\dim(G)=k$ for any $G\in \calF(\Delta_{\bdalpha}(I_{\Delta}^{(n)}))$. If the pure simplicial complex $\Delta_{\bdalpha}(I_{\Delta}^{(n)})|_B$ is not a sphere, then it is collapsible and consequently $\widetilde{H}_{k} (\Delta_{\bdalpha}(I_{\Delta}^{(n)})) = 0$, contradicting to our assumption. Hence $\Delta|_B=\Delta_{\bdalpha}(I_{\Delta}^{(n)})|_B$ is indeed a sphere.

    \textbf{\ref{Theo4.9-b}$\Rightarrow$\ref{Theo4.9-a}:} Without loss of generality, we may assume that $B=[k+2]$ and $\Delta|_B$ is a $k$-dimensional sphere. Take $\bdalpha=(\alpha_1,\dots,\alpha_s)\in \NN^s$ where $\alpha_i=n-1$ for $0 \leq i \leq k+2$ and $\alpha_i=0$ for $k+2 < i \leq s$. Then by \Cref{3.2}, we have $\Delta|_B=\Delta_{\bdalpha}(I_{\Delta}^{(n)})$ and $\widetilde{H}_{k} (\Delta_{\bdalpha}(I_{\Delta}^{(n)})) \neq 0$. So $H_{\mathfrak{m}}^{k+1}(S / I_{\Delta}^{(n)})_{\bdalpha} \neq 0$ via \Cref{3.1}, which implies that $\alpha_{k+1}(S/ I_{\Delta}^{(n)})\geq(k+2)(n-1)$. Since $a_{k+1}(S/ I_{\Delta}^{(n)}) \leq (k+2)(n-1)$ in general, the proof is completed.  
\end{proof}

\begin{acknowledgment*}
    The second author is partially supported by the ``Anhui Initiative in Quantum Information Technologies'' (No.~AHY150200) and the ``Fundamental Research Funds for the Central Universities''.
\end{acknowledgment*}


\end{document}